\documentclass[a4paper,pdftex]{amsart}
\usepackage{amsmath,amssymb,amsbsy,amsthm,amsfonts,comment,tikz}
\usepackage[all]{xy}
\usepackage{graphicx,color}
\usepackage{hyperref}
\hypersetup{
    colorlinks=true,
    linkcolor=purple,
    urlcolor=orange,
    citecolor=teal,
    pdfauthor = {Ryuma Orita},
    pdftitle = {Topological complexity of monotone symplectic manifolds},
    pdfpagemode = UseNone,
    bookmarksnumbered=true,
}

\title{Topological complexity of monotone symplectic manifolds}

\author{Ryuma Orita} 
\address{Department of Mathematics, Faculty of Science, Niigata University, Niigata 950-2181, Japan}
\email{\href{mailto:orita@math.sc.niigata-u.ac.jp}{orita@math.sc.niigata-u.ac.jp}}
\urladdr{\url{https://ryuma-orita.netlify.app/}}

\thanks{This work was supported by JSPS KAKENHI Grant Number 21K13787.}

\newtheorem{theorem}{Theorem}[section]

\newtheorem{proposition}[theorem]{Proposition}
\newtheorem{corollary}[theorem]{Corollary}

\theoremstyle{definition}
\newtheorem{definition}[theorem]{Definition}

\theoremstyle{remark}
\newtheorem{remark}[theorem]{Remark}

\newcommand{\ZZ}{\mathbb{Z}}

\newcommand{\RR}{\mathbb{R}}
\newcommand{\CC}{\mathbb{C}}

\newcommand{\Image}{\operatorname{Im}}

\newcommand{\PD}{\mathrm{PD}}
\newcommand{\pr}{\mathrm{pr}}
\newcommand{\cat}{\mathsf{cat}}
\newcommand{\TC}{\mathsf{TC}}

\newcommand{\wgt}{\mathrm{wgt}}
\newcommand{\cwgt}{\mathrm{cwgt}}

\begin{document}

\begin{abstract}
We study Farber's topological complexity for monotone symplectic manifolds.
More precisely, we estimate the topological complexity of $4$-dimensional spherically monotone manifolds whose Kodaira dimension is not $-\infty$.
\end{abstract}

\maketitle

\tableofcontents


\section{Introduction and main results}\label{sec:introduction}

\subsection{Topological complexity}

In his celebrated papers \cite{Fa03,Fa04}, Farber introduced a homotopy invariant $\TC(X)$ of a topological space $X$,
which detects a topological complexity of algorithms for autonomous robot motion planning in the configuration space $X$ of a mechanical system.
To be more precise, Farber defined the \textit{topological complexity} $\TC(X)$ of $X$
to be the smallest integer $k$ (or infinity) such that the product $X\times X$ can be covered by $k$ open subsets $U_1,\ldots,U_k$ on each of which the free path fibration
\[
	p\colon X^I\to X\times X;\quad \gamma\mapsto (\gamma(0),\gamma(1))
\]
admits a continuous section $U_i\to X^I$, where $I=[0,1]$.
For example, Farber \cite[Theorem 9]{Fa03} computed the topological complexity of a closed orientable surface $\Sigma_g$ of genus $g$ as follows:
\[
    \TC(\Sigma_g)=\begin{cases}%
        3 & \text{if $g=0,1$}, \\
        5 & \text{if $g\geq 2$}.
    \end{cases}
\]
It is easy to see that a path-connected space $X$ satisfies $\TC(X)=1$ if and only if $X$ is contractible.
We refer the reader to \cite{Fa06} and \cite[Chapter 4]{Fa08} for more examples and an excellent exposition.
We note that in several references, a \textit{normalized} version is used, which is one less than our $\TC$.

A general upper bound for $\TC$ is achieved by the connectivity of the space $X$.
Namely, if $X$ is an $r$-connected (i.e., $\pi_i(X)$ vanishes for all $i\leq r$) CW-complex, then
\begin{equation}\label{eq:upper_bound}
    \TC(X)\leq\frac{2\dim{X}}{r+1}+1,
\end{equation}
see \cite[Theorem 5.2]{Fa03} for the proof.
For example, if $X$ is (path-)connected, then $\TC(X)\leq 2\dim{X}+1$.
Moreover if $X$ is simply connected, then $\TC(X)\leq \dim{X}+1$.

A \textit{symplectic manifold} is a pair $(M,\omega)$ consisting of a $2n$-dimensional smooth manifold $M$ equipped with a non-degenerate closed 2-form $\omega$.
The non-degeneracy of $\omega$ means that its $n$-th exterior power $\omega^n$ is nowhere vanishing,
and hence its cohomology class $[\omega]^n\in H^{2n}(M;\RR)$ is non-zero.
The existence of such a cohomology class helps us to estimate $\TC$ from below, see Section \ref{sec:TC-weight} for details.

\begin{theorem}[{\cite[Corollary 3.2]{FTY03}}]\label{thm:simply_conn}
    Let $(M,\omega)$ be a simply connected closed symplectic manifold.
    Then $\TC(M)=\dim{M}+1$.
\end{theorem}

\textit{
    Throughout the paper, we assume that our symplectic manifolds are connected and non-simply connected unless otherwise noted.
}

The topological complexity of a connected CW-complex $X$ is called \textit{maximal} \cite{CV21} if $\TC(X)=2\dim X+1$ (compare to \eqref{eq:upper_bound}).
For example, $\TC$ of a closed orientable surface of genus $g\geq 2$ is maximal.
Recently, it is shown that $\TC$ of a closed non-orientable surface of genus $g\geq 2$ is also maximal
by Dranishnikov \cite{Dr16} for $g\geq 4$, and by Cohen and Vandembroucq \cite{CV17} for $g=2,3$.

Let $(M,\omega)$ be a closed symplectic manifold.
Grant and Mescher \cite{GM20} showed that if $M$ admits a symplectic form whose cohomology class $[\omega]$ is atoroidal (see Definition \ref{def:atoroidal_aspherical} (ii) for the definition),
then $\TC(M)$ is maximal.
For example, every closed orientable surface $\Sigma_g$ of genus $g\geq 2$ admits an atoroidal symplectic form.

\begin{theorem}[{\cite[Theorem 1.2]{GM20}}]\label{thm:atoroidal}
    Let $(M,\omega)$ be a closed symplectic manifold with atoroidal $\omega$.
    Then $\TC(M)=2\dim{M}+1$.
\end{theorem}

An important property of $\TC$ is its homotopy invariance \cite[Theorem 3]{Fa03}.
Therefore, the topological complexity of an aspherical space $X$ (i.e., $\pi_2(X)=0$ for all $i\geq 2$) depends only on its fundamental group $\pi=\pi_1(X)$.
To keep this in mind, the topological complexity of a group $\pi$ is defined to be that of the Eilenberg--MacLane space $K(\pi,1)$: $\TC(\pi)=\TC(K(\pi,1))$.
For example, the circle $S^1$ is of $K(\ZZ,1)$ and $\TC(\ZZ)=2$ by direct computation.
In fact, Grant, Lupton, and Oprea \cite{GLO13} showed that $\ZZ$ is the only group having $\TC(\pi)=2$.
More generally, the $n$-torus $T^n$ is of $K(\ZZ^n,1)$.
Farber \cite[Theorem 12]{Fa03} showed that $\TC(T^n)=\dim{T^n}+1$.
Thus $\TC(\ZZ^n)=\TC(T^n)$ is far from being maximal.
Roughly speaking, an abelian group might be too simple to have large $\TC$.
Cohen and Vandembroucq \cite{CV21} studied the topological complexity of abelian groups.
We refer the reader to \cite{CF10,GLO15b} for other results concerning the topological complexity of aspherical spaces.

However, Farber and Mescher \cite{FM20} showed that
the topological complexity of an aspherical space can also attain either its maximum or one less than the maximum
if one imposes a condition on the fundamental group.
Let us now consider a symplectic manifold $(M,\omega)$.
Even if it is not aspherical as a space, it sometimes admits an \textit{aspherical} symplectic form (see Definition \ref{def:atoroidal_aspherical} (i)).
This is a weaker condition than being atoroidal.

\begin{theorem}[{\cite[Theorem 10.1]{FM20}}]\label{thm:aspherical}
    Let $(M,\omega)$ be a closed symplectic manifold with aspherical $\omega$.
    Assume that the fundamental group $\pi_1(M)$ is of type FL and the centralizer of every non-trivial element is infinitely cyclic.
    Then $\TC(M)$ is either $2\dim{M}$ or $2\dim{M}+1$.
\end{theorem}

A group $\pi$ is of \textit{type FL} if $\ZZ$ admits a finite free resolution over the group ring $\ZZ[\pi]$, see \cite[Section VIII.6]{Br82}.
Computing $\TC$ for all aspherical spaces $K(\pi,1)$ has not been done yet in general.
Grant \cite{Gr12} gave upper bounds for $\TC(\pi)$ for nilpotent groups $\pi$.
Grant, Lupton, and Oprea \cite{GLO15a} studied lower bounds.
Moreover, Farber, Grant, Lupton, and Oprea \cite{FGLO19} studied upper and lower bounds in terms of Bredon cohomology.

In this paper, we endeavor to extend the conditions beyond aspherical.
As an even weaker condition than being aspherical, there is the condition of being spherically monotone (see Definition \ref{def:monotone} (iii)).
Let us now state our main result.

\begin{theorem}\label{thm:main}
    Let $(M,\omega)$ be a closed 4-dimensional symplectic manifold whose Kodaira dimension is not $-\infty$.
    \begin{enumerate}
        \item If $(M,\omega)$ is toroidally monotone, then $\TC(M)=9$.
        \item If $(M,\omega)$ is spherically monotone and the fundamental group $\pi_1(M)$ is of type FL and every non-trivial element of the center is infinitely cyclic, then $\TC(M)$ is either $8$ or $9$.
    \end{enumerate}
\end{theorem}

We refer the reader to Section \ref{sec:preliminaries_symplectic} for the definitions of the notions appearing in Theorem \ref{thm:main}.
Unfortunately, we do not have any specific examples meeting the requirements of Theorem \ref{thm:main} except Theorem \ref{thm:aspherical}.
Instead, we provide some remarks.

\begin{remark}
   It is shown by Gompf \cite[Theorem 0.1]{Go95} that every finitely presentable group can be realized as the fundamental group of a closed 4-dimensional symplectic manifold.
\end{remark}

\begin{remark}
    According to Li--Liu \cite{LL95}, Liu \cite{Liu96}, and Ohta--Ono \cite{OO96a,OO96b},
    every 4-dimensional symplectic manifold with Kodaira dimension $-\infty$ is diffeomorphic to a symplectic blow up of a rational or ruled symplectic manifold.
    That is, it is a blow up of either $S^2\times S^2$, $\CC P^2$, or an $S^2$-bundle over a closed oriented surface $\Sigma_g$ of genus $g$.
    We note that $\CC P^2$ and the total space of an $S^2$-bundle over $S^2$ is simply connected,
    which are excluded from our symplectic manifolds from the beginning.
\end{remark}

\begin{remark}
    One can construct toroidally monotone or spherically monotone symplectic manifolds as follows.
    Let $(M_1,\omega_1)$ be a closed \textit{symplectically atoroidal} (i.e., both $[\omega_1]$ and $c_1(M_1,\omega_1)$ are atoroidal) symplectic manifold.
    For example, every K\"ahler manifold with negative sectional curvature (e.g., $\Sigma_{g\geq 2}$) is symplectically atoroidal, see also \cite{Gu13} for further discussion.
    Let $(M_2,\omega_2)$ be a closed \textit{symplectically aspherical} (i.e., both $[\omega_2]$ and $c_1(M_2,\omega_2)$ are aspherical) symplectic manifold.
    There are various constructions of such manifolds due to \cite{Go98,IKRT04,KRT07}.
    Next, let $(N,\omega_N)$ be a closed strongly monotone symplectic manifold with monotonicity constant $\lambda_N$.
    For instance, consider a smooth complete intersection $N$ in the complex projective space $\mathbb{C}P^{n+k}$ defined by $k$ homogeneous polynomials of degrees $d_1,\ldots,d_k$ where $n>k$.
    If $n+k+1>d_1+\cdots+d_k$, then the manifold $N$ equipped with the symplectic form $\omega_N$ induced from the K\"ahler form of $\mathbb{C}P^{n+k}$ is strongly monotone with $\lambda_N=\bigr(n+k+1-(d_1+\cdots+d_k)\bigr)^{-1}$, see \cite[Example 2.8 (p.~88)]{LM89}.
    Now it is straightforward to see that the product $M_1\times N$ (resp.\ $M_2\times N$) is toroidally monotone (resp.\ spherically monotone) with monotonicity constant $\lambda_N$.
\end{remark}

We prove Theorem \ref{thm:main} in Section \ref{sec:proof}.


\subsection{Lusternik--Schnirelmann category}

The Lusternik--Schnirelmann (LS) category $\cat(X)$ of a path-connected topological space $X$ is another invariant related to $\TC$.
The LS category $\cat(X)$ is defined to be the smallest integer $k$ (or infinity) such that $X$ can be covered by $k$ open subsets $U_1,\ldots,U_k$ each of which is contractible in $X$.
Similarly to $\TC$, it should be noted that for the LS category,
many references use a \textit{normalized} version which is one less than ours.
The connectivity of $X$ provides an upper bound for $\cat$.
That is, if $X$ is an $r$-connected CW-complex, then
\begin{equation}\label{eq:upper_bound_LS}
    \cat(X)\leq\frac{\dim{X}}{r+1}+1,
\end{equation}
see \cite[Theorem 1.50]{CLOT03} for the proof.
Farber \cite[Theorem 5]{Fa03} showed that the LS category gives upper and lower bounds for the topological complexity:
\[
    \cat(X) \leq \TC(X) \leq \cat(X\times X) \leq 2\dim X+1.
\]
We mention that these notions are special cases of the \textit{sectional category} (or \textit{Schwarz genus}) of a fibration \cite{Sch66}.

Now let us collect the results for the LS category for symplectic manifolds $(M,\omega)$.
If $(M,\omega)$ is simply connected and $\dim M=2n$,
then the inequality \eqref{eq:upper_bound_LS} implies that $\cat(M)\leq n+1$.
On the other hand, since $\omega$ is symplectic, the cohomology class $[\omega]^n\in H^{2n}(M;\RR)$ is non-zero;
the existence of such a cohomology class helps us to estimate $\cat$ from below, see Section \ref{sec:category_weight} for details.
In this case, $\cat(M)\geq n+1$.
Therefore, $\cat(M)=n+1$.
So we are interested in \textit{non}-simply connected $(M,\omega)$.

\begin{theorem}[{\cite[Corollary 4.2]{RO99}}]\label{thm:aspherical_LS}
    Let $(M,\omega)$ be a closed symplectic manifold with aspherical $\omega$.
    Then $\cat(M)=\dim{M}+1$.
\end{theorem}

Theorem \ref{thm:aspherical_LS} provided a different proof of the Arnold conjecture for such a manifold being true, see \cite[Corollary 4.3]{RO99}.

Let us state another result concerning the LS category for spherically monotone symplectic manifolds.

\begin{theorem}\label{thm:main_LS}
    Let $(M,\omega)$ be a closed 4-dimensional spherically monotone symplectic manifold whose Kodaira dimension is not $-\infty$.
    Then $\cat(M)=5$.
\end{theorem}

We show Theorem \ref{thm:main_LS} in Section \ref{sec:proof}.


\section{Symplectic manifolds}\label{sec:preliminaries_symplectic}

In this section, we recall the definition of ``monotone'' symplectic manifolds and the Kodaira dimension of symplectic 4-manifolds.

\subsection{Monotone symplectic manifolds}\label{subsection:monotonicity}

Let $X$ be a connected CW-complex and $\mathcal{L}X$ denote the space of free loops in $X$.
Given a free homotopy class $\alpha\in [S^1,X]$,
we define $\mathcal{L}_{\alpha}X$ to be the component of $\mathcal{L}X$ with loops representing $\alpha$.

Every element of $\pi_2(X)$ is represented by a map $w\colon S^2\to X$.
Let
\[
	h_S\colon\pi_2(X)\to H_2(X;\mathbb{Z});\quad [w]\mapsto w_*([S^2])
\]
denote the Hurewicz homomorphism,
where $[S^2]$ is the fundamental class of $S^2$
and $w_{\ast}\colon H_2(S^2;\mathbb{Z})\to H_2(X;\mathbb{Z})$ is an induced homomorphism.
Similarly, every element of $\pi_1(\mathcal{L}_{\alpha}X)$ is represented by
a map $w\colon S^1\times S^1\to X$.
We define a homomorphism
\[
	h_T\colon\pi_1(\mathcal{L}_{\alpha}X)\to H_2(X;\mathbb{Z});\quad [w]\mapsto w_*([S^1\times S^1]).
\]

\begin{definition}\label{def:atoroidal_aspherical}
    Let $u\in H^2(X;R)$ where $R=\RR$ or $\ZZ$.
    \begin{enumerate}
        \item $u$ is called \textit{aspherical} if $u$ vanishes on $h_S(\pi_2(X))$.
        \item $u$ is called \textit{atoroidal} if $u$ vanishes on $h_T(\pi_1(\mathcal{L}_{\alpha}X))$ for all $\alpha\in [S^1,X]$.
    \end{enumerate}
\end{definition}

Let $(M,\omega)$ be a connected closed symplectic manifold.
A closed 2-form $\eta$ on $M$ is said to be \textit{aspherical} (resp.\ \textit{atoroidal})
if its cohomology class $[\eta]\in H^2(M;\RR)$ is aspherical (resp.\ atoroidal).

Let us consider the tangent bundle $TM$ of $M$.
An endomorphism $J\colon TM\to TM$ is called an \textit{almost complex structure} if it satisfies $J^2=-\mathrm{id}_{TM}$.
An almost complex structure $J$ is said to be \textit{compatible with $\omega$} if the bilinear form $\omega(\cdot,J\cdot)$ defines a Riemannian metric on the tangent bundle $TM$.
Such a $J$ makes the tangent bundle $TM$ into a complex vector bundle.
Let $c_1=c_1(M,\omega)\in H^2(M;\ZZ)$ denote the first Chern class of the complex vector bundle $(TM,J)$.
It is known that this definition does not depend on the choice of the almost complex structure $J$ compatible with $\omega$,
see, e.g., \cite[Section 4.1]{MS17}.

\begin{definition}\label{def:monotone}
    Let $(M,\omega)$ be a connected closed symplectic manifold.
    \begin{enumerate}
        \item $(M,\omega)$ is called \textit{strongly monotone} (resp.\ \textit{strongly negative monotone})
        if there exists a positive (resp.\ negative) number $\lambda\in\mathbb{R}$ such that
        \[
    	[\omega] = \lambda c_1.
        \]
        \item $(M,\omega)$ is called \textit{toroidally monotone} (resp.\ \textit{toroidally negative monotone})
        if there exists a non-negative (resp.\ negative) number $\lambda\in\mathbb{R}$ such that for all $\alpha\in [S^1,M]$,
        \[
    	[\omega]|_{h_T(\pi_1(\mathcal{L}_{\alpha}M))} = \lambda c_1|_{h_T(\pi_1(\mathcal{L}_{\alpha}M))}.
        \]
        \item $(M,\omega)$ is called \textit{spherically monotone} (resp.\ \textit{spherically negative monotone})
        if there exists a non-negative (resp.\ negative) number $\lambda\in\mathbb{R}$ such that
        \[
    	[\omega]|_{h_S(\pi_2(M))} = \lambda c_1|_{h_S(\pi_2(M))}.
        \]
    \end{enumerate}
\end{definition}

\begin{remark}
    We note that the term \textit{monotone symplectic manifold} can mean \textit{spherically} monotone
    or \textit{strongly} monotone symplectic manifold depending on the literature.
    We will summarize the relationships of the above concepts in the following diagram:
    \[
        \xymatrix{
        & \text{atoroidal} \ar@{=>}[r] \ar@{=>}[d] & \text{aspherical} \ar@{=>}[d] \\
        \text{strongly monotone} \ar@{=>}[r] & \text{toroidally monotone} \ar@{=>}[r] & \text{spherically monotone} \\
        }
    \]
\end{remark}

Given a group $\pi$,
we recall that the \textit{Eilenberg--MacLane space} $K(\pi,1)$ is defined to be a connected CW-complex with fundamental group $\pi$ such that $\pi_i(K(\pi,1))=0$ for any $i\geq 2$.
The following proposition is useful for our purpose.

\begin{proposition}[{\cite[Lemma 2.1]{RT99}}]\label{prop:aspherical_iff}
Let $X$ be a finite CW-complex and $u\in H^2(X;\RR)$ an aspherical cohomology class.
Then for every map $f\colon X\to K(\pi_1(X),1)$ which induces an isomorphism of fundamental groups,
\[
	u\in\Image\left(f^{\ast}\colon H^2(K(\pi_1(X),1);\RR)\to H^2(X;\RR)\right).
\]
\end{proposition}

\subsection{Symplectic 4-manifolds}

Let $(M,\omega)$ be a connected closed 4-dimensional symplectic manifold.
We choose an orientation of $M$ compatible with $\omega$.
Let $Q_M\colon H^2(M;R)\times H^2(M;R)\to R$ denote the intersection form of $M$, where $R=\RR$ or $\ZZ$.
That is, for $u,v\in H^2(M;R)$ we define
\[
    Q_M(u,v) = u\cdot v = \int_M u \smile v.
\]

\begin{definition}\label{def:rational_ruled}
    A closed symplectic 4-manifold is called \textit{minimal} if it contains no symplectically embedded 2-sphere of self-intersection number $-1$.
    Such a 2-sphere is called an \textit{exceptional sphere}.
\end{definition}

By symplectic blowing down \cite{Mc91} exceptional spheres, one can get a decomposition into a connected sum
\[
    M=M_{\mathrm{min}}\#k\overline{\CC P^2},
\]
for some $k\geq 0$ and a minimal symplectic 4-manifold $M_{\mathrm{min}}$,
where $\overline{\CC P^2}$ denotes $\CC P^2$ with opposite orientation.
We call $M_{\mathrm{min}}$ a \textit{minimal model} of $M$.
McDuff \cite{Mc90} showed that the minimal model is unique up to symplectomorphism unless $(M,\omega)$ is rational or ruled in the following sense.

\begin{definition}
    Let $(M,\omega)$ be a closed symplectic 4-manifold.
    \begin{enumerate}
        \item $(M,\omega)$ is \textit{rational} if it is diffeomorphic to $S^2\times S^2$ or a blowup of $\CC P^2$.
        \item $(M,\omega)$ is \textit{ruled} if it is diffeomorphic to a blowup of an $S^2$-bundle over a closed oriented surface $\Sigma_g$.
    \end{enumerate}
\end{definition}

Let $K=-c_1(M,\omega)\in H^2(M;\ZZ)$ denote the \textit{canonical class} of $(M,\omega)$.
This is the first Chern class of the cotangent bundle $T^*M$, considered as a complex vector bundle with respect to some almost complex structure compatible with $\omega$.
Now we define the Kodaira dimension of $(M,\omega)$ (see also \cite[Section 13.4]{MS17}).
The notion of Kodaira dimension was originally defined for smooth projective varieties.

\begin{definition}
    Let $(M,\omega)$ be a minimal symplectic 4-manifold.
    The \textit{Kodaira dimension} of $(M,\omega)$ is defined to be
    \[
    	\kappa(M,\omega)=%
    	\begin{cases}
    		-\infty & \text{if}\ K\cdot K<0\ \text{or}\ K\cdot[\omega]<0,\\
    		0 & \text{if}\ K\cdot K=0\ \text{and}\ K\cdot[\omega]=0,\\
    		1 & \text{if}\ K\cdot K=0\ \text{and}\ K\cdot[\omega]>0,\\
    		2 & \text{if}\ K\cdot K>0\ \text{and}\ K\cdot[\omega]>0,
    	\end{cases}
    \]
    see \cite{MS96,Li06,Li08}.
\end{definition}

According to \cite[Lemma 2.5]{Li06}, the conditions $K\cdot [\omega]=0$ and $K\cdot K\geq 0$ implies that $K\cdot K=0$,
and hence $\kappa(M,\omega)$ is well defined.
The Kodaira dimension of non-minimal $(M,\omega)$ is defined to be that of its minimal models.
Li \cite[Theorem 2.6]{Li06} showed that it is independent of the choice of its minimal models and a diffeomorphism invariant.

According to Li--Liu \cite{LL95}, Liu \cite{Liu96}, and Ohta--Ono \cite{OO96a,OO96b},
a closed symplectic 4-manifold has Kodaira dimension $-\infty$ if and only if it is a blowup of a rational or ruled symplectic 4-manifold.
Li \cite{Li06} studied symplectic 4-manifolds with $\kappa(M,\omega)=0$.
For example, minimal symplectic 4-manifolds with Kodaira dimension zero include K3 surfaces, Enriques surface, and $T^2$-bundles over $T^2$ such as the Kodaira--Thurston manifold.
Baldridge and Li \cite{BL05} studied symplectic 4-manifolds with $\kappa(M,\omega)=1$.
Symplectic 4-manifolds with $\kappa(M,\omega)=2$ are said to have \textit{general type}.

We note that every strongly monotone symplectic manifold $(M,\omega)$ has Kodaira dimension $-\infty$.
Indeed, the condition $[\omega]=\lambda c_1$ for some $\lambda > 0$ implies that
\[
    K\cdot [\omega] = -c_1\cdot [\omega] = -\frac{1}{\lambda} [\omega]\cdot[\omega] = -\frac{1}{\lambda} \int_M \omega\wedge\omega < 0
\]
since $\omega$ is symplectic.
Actually, Ohta and Ono \cite{OO96b} proved that every strongly monotone symplectic manifold is diffeomorphic to either $S^2\times S^2$
or $X_k=\CC P^2\#k\overline{\CC P^2}$ for $0\leq k \leq 8$,
see also Li--Liu \cite{LL95}, McDuff \cite{Mc90}, and Taubes \cite{Ta95,Ta96,Ta00}.


\section{Lower bounds for TC and LS category}

In this section, we introduce cohomological weights that are useful for evaluations from below of the topological complexity and the Lusternik--Schnirelmann category.

\subsection{TC-weight}\label{sec:TC-weight}

We define $\TC$-weight introduced by Farber and Grant \cite{FG07}, which generalizes the notion of category weight \cite{FH92} (see Section \ref{sec:category_weight} for the latter).
Let $X$ be a path-connected topological space
and $A$ a local coefficient system on $X\times X$.

\begin{definition}[\cite{FG07,FG08}]\label{def:TC-weight}
    The \textit{$\TC$-weight} of a cohomology class $u\in H^*(X\times X;A)$ is defined to be
    the largest integer $k$ with the property that $f^*u=0\in H^*(Y;f^{\ast}A)$ for any continuous map $f\colon Y\to X\times X$,
    where $Y$ is a topological space which has an open cover $\{U_1,\ldots,U_k\}$ such that each $U_i$ admits a continuous lift $s_i\colon U_i\to X^I$
    with $p\circ s_i=f|_{U_i}$:
    \[
        \xymatrix{
        & X^I \ar[d]^-p \\
        U_i \ \ar[r]_-{f|_{U_i}} \ar[ru]^-{s_i} & X\times X \\
        }
    \]
    Let $\wgt(u)$ denote the $\TC$-weight of $u$.
\end{definition}

We note that the $\TC$-weight of 0 is $\infty$.
By definition, if $X$ admits a non-zero cohomology class $u$ such that $\wgt(u)\geq k-1$, then we have $\TC(X)\geq k$.
Moreover, $\TC$-weight has the following important property:

\begin{theorem}[{\cite[Proposition 32]{FG07}}]\label{thm:TC-weight_cup_product}
    Let $A$ and $B$ be local coefficient systems on $X\times X$.
    Let $u\in H^i(X\times X;A)$ and $v\in H^j(X\times X;B)$ be cohomology classes.
    Then the cup product $u\smile v\in H^{i+j}(X\times X;A\otimes_{\ZZ}B)$ satisfies
    \[
        \wgt(u\smile v) \geq \wgt(u) + \wgt(v).
    \]
\end{theorem}

\begin{definition}
    A cohomology class $u\in H^*(X\times X;A)$ is called a \textit{zero-divisor} if $u|_{\Delta_X}=0\in H^*(\Delta_X;A|_{\Delta_X})$,
    where $\Delta_X$ is the diagonal of $X\times X$.
\end{definition}

Farber and Grant \cite[Section 3]{FG08} showed that a cohomology class $u$ is a zero-divisor if and only if $\wgt(u)\geq 1$.

\subsection{Category weight}\label{sec:category_weight}

We define category weight introduced by Fadell and  Husseini \cite{FH92} and developed by Rudyak \cite{Ru99}.
Let $X$ be a path-connected topological space and $A$ an abelian group.

\begin{definition}[\cite{FH92}]\label{def:category_weight}
    The \textit{category weight} of a cohomology class $u\in H^*(X;A)$ is defined to be
    the largest integer $k$ with the property that $u|_Y=0\in H^*(Y;A)$ for any subspace $Y\subset X$
    which is covered by $k$ sets open in $X$ and contractible in $X$.
    Let $\cwgt(u)$ denote the category weight of $u$.
\end{definition}

We note that the category weight of 0 is $\infty$.
By definition, if $X$ admits a non-zero cohomology class $u$ such that $\cwgt(u)\geq k-1$, then we have $\cat(X)\geq k$.
Moreover, category weight has the following important property:

\begin{theorem}[{\cite[Theorem 1.1]{FH92}}]\label{thm:cateogry_weight_cup_product}
    Let $u\in H^i(X;A)$ and $v\in H^j(X;A)$ be cohomology classes.
    Then the cup product $u\smile v\in H^{i+j}(X;A)$ satisfies
    \[
        \cwgt(u\smile v) \geq \cwgt(u) + \cwgt(v).
    \]
\end{theorem}


\section{Proofs}\label{sec:proof}

In this section, we prove our main results (Theorems \ref{thm:main} and \ref{thm:main_LS}).
Given a set $X$, let $\pr_1,\pr_2\colon X\times X\to X$ denote the first and second projections, respectively.

\subsection{Proof of Theorem \ref{thm:main}}

We rely on the following estimates for $\TC$-weight given by Grant--Mescher \cite{GM20}, and Farber--Mescher \cite{FM20}.

\begin{theorem}[{\cite[Corollary 3.6 and Proof of Theorem 1.2]{GM20}}]\label{thm:atoroidal_TC-wgt}
    Let $(M,\omega)$ be a closed symplectic manifold with atoroidal $\omega$.
    Then the cohomology class $[\pr_2^*\omega-\pr_1^*\omega]\in H^2(M\times M;\RR)$
    has $\TC$-weight greater than or equal to 2:
    $\wgt([\pr_2^*\omega-\pr_1^*\omega])\geq 2$.
\end{theorem}

\begin{theorem}[{\cite[Theorem 3]{FM20}}]\label{thm:aspherical_TC-wgt}
    Let $X$ be a connected aspherical finite cell complex and $A$ a local coefficient system on $X\times X$.
    Let $u\in H^k(X\times X;A)$ be a zero-divisor, where $k\geq 1$.
    Assume that the centralizer of every non-trivial element of the fundamental group $\pi_1(X)$ is infinite cyclic.
    Then $\wgt(u)\geq k-1$.
\end{theorem}

In view of Theorems \ref{thm:atoroidal_TC-wgt} and \ref{thm:aspherical_TC-wgt},
we can generalize Theorems \ref{thm:atoroidal} and \ref{thm:aspherical} as follows:

\begin{theorem}\label{thm:main_general_toroidally}
    Let $(M,\omega)$ be a $2n$-dimensional toroidally monotone or toroidally negative monotone symplectic manifold with monotonicity constant $\lambda$.
    If $([\omega]-\lambda c_1)^n\in H^{2n}(M;\RR)$ is non-zero, then $\TC(M)=4n+1$.
\end{theorem}

\begin{theorem}\label{thm:main_general_spherically}
    Let $(M,\omega)$ be a $2n$-dimensional spherically monotone or spherically negative monotone symplectic manifold with monotonicity constant $\lambda$.
    Assume that the fundamental group $\pi_1(M)$ is of type FL and every non-trivial element of the center is infinitely cyclic.
    If $([\omega]-\lambda c_1)^n\in H^{2n}(M;\RR)$ is non-zero, then $\TC(M)$ is either $4n$ or $4n+1$.
\end{theorem}

For example, direct products of manifolds meeting the requirements of Theorem \ref{thm:main} satisfy the conditions of Theorem \ref{thm:main_general_toroidally} or \ref{thm:main_general_spherically}.
The proofs of Theorems \ref{thm:main_general_toroidally} and \ref{thm:main_general_spherically} are very strongly influenced by \cite{GM20,FM20}.

\begin{proof}[Proof of Theorem \ref{thm:main_general_toroidally}]
If $(M,\omega)$ is toroidally monotone or toroidally negative monotone with monotonicity constant $\lambda$,
then there exists $\lambda\in\RR$ such that for any $\alpha\in [S^1,M]$,
\[
    [\omega]|_{h_T(\pi_1(\mathcal{L}_{\alpha}M))} = \lambda c_1|_{h_T(\pi_1(\mathcal{L}_{\alpha}M))}.
\]
That is,
\[
    ([\omega] - \lambda c_1)|_{h_T(\pi_1(\mathcal{L}_{\alpha}M))} = 0.
\]
It implies that the cohomology class $[\omega] - \lambda c_1$ is atoroidal.
Let $\eta$ be a closed 2-form on $M$ that represents $[\omega] - \lambda c_1\in H^2(M;\RR)$.
Although $\eta$ need not be symplectic, one can apply Theorem \ref{thm:atoroidal_TC-wgt} for the atoroidal closed 2-form $\eta$
(see also the proof of \cite[Corollary 3.6]{GM20}).
Therefore, we obtain the estimate $\wgt([\pr_2^*\eta-\pr_1^*\eta])\geq 2$.
By Theorem \ref{thm:TC-weight_cup_product},
\[
    \wgt([\pr_2^*\eta-\pr_1^*\eta]^{2n}) \geq 2n\cdot\wgt([\pr_2^*\eta-\pr_1^*\eta]) \geq 4n.
\]

By the assumption, $[\eta]^n=([\omega]-\lambda c_1)^n\in H^{2n}(M;\RR)$ is non-zero.
Therefore, the cohomology class $[\pr_2^*\eta-\pr_1^*\eta]^{2n}$ is also non-zero
since it contains the non-zero summand
\[
    (-1)^n\binom{2n}{n}\pr_2^*[\eta]^n\smile\pr_1^*[\eta]^n = (-1)^n\binom{2n}{n}[\eta]^n\times [\eta]^n.
\]
Hence $\TC(M)\geq 4n+1$.
By \eqref{eq:upper_bound}, we conclude that $\TC(M) = 4n+1$.
\end{proof}

\begin{proof}[Proof of Theorem \ref{thm:main_general_spherically}]
If $(M,\omega)$ is spherically monotone or spherically negative monotone with monotonicity constant $\lambda$,
then there exists $\lambda\in\RR$ such that
\[
    [\omega]|_{h_S(\pi_2(M))} = \lambda c_1|_{h_S(\pi_2(M))}.
\]
That is,
\[
    ([\omega] - \lambda c_1)|_{h_S(\pi_2(M))} = 0.
\]
It implies that the cohomology class $[\omega] - \lambda c_1$ is aspherical.

Since $\pi_1(M)$ is of type FL, there exists a finite CW-complex $K=K(\pi_1(M),1)$, see \cite[Theorem VIII.7.1]{Br82}.
Let $f\colon M\to K$ be a continuous map which induces an isomorphism of fundamental groups.
Applying Proposition \ref{prop:aspherical_iff} for $[\omega] - \lambda c_1$ yields that
there exists a cohomology class $\Omega\in H^2(K;\RR)$ such that $f^*\Omega = [\omega] - \lambda c_1$.

Let us consider the cohomology class $(\pr_2^*\Omega-\pr_1^*\Omega)^{2n} \in H^{4n}(K\times K;\RR)$.
It is a zero-divisor since $(\pr_2^*\Omega-\pr_1^*\Omega)^{2n}|_{\Delta_K}=(\Omega-\Omega)^{2n}=0$.
Applying Theorem \ref{thm:aspherical_TC-wgt} for $(\pr_2^*\Omega-\pr_1^*\Omega)^{2n}$,
we obtain that the estimate $\wgt\bigl((\pr_2^*\Omega-\pr_1^*\Omega)^{2n}\bigr)\geq 4n-1$.

Now consider the zero-divisor $\bigl(\pr_2^*([\omega]-\lambda c_1)-\pr_1^*([\omega]-\lambda c_1)\bigr)^{2n} \in H^{4n}(M\times M;\RR)$.
Let $g\colon Y\to M\times M$ be a continuous map,
where $Y$ is a topological space which has an open cover $\{U_1,\ldots,U_{4n-1}\}$ such that each $U_i$ admits a continuous lift $s_i\colon U_i\to M^I$ with $p\circ s_i=g|_{U_i}$ (recall Definition \ref{def:TC-weight}).
Then
\begin{align*}
    g^*\left(\bigl(\pr_2^*([\omega]-\lambda c_1)-\pr_1^*([\omega]-\lambda c_1)\bigr)^{2n}\right)%
    & = g^*\bigl((f\times f)^*(\pr_2^*\Omega-\pr_1^*\Omega)^{2n}\bigr) \\
    & = \bigl((f\times f)\circ g\bigr)^*\bigl((\pr_2^*\Omega-\pr_1^*\Omega)^{2n}\bigr) = 0.
\end{align*}
Thus $\wgt \bigl(\pr_2^*([\omega]-\lambda c_1)-\pr_1^*([\omega]-\lambda c_1)\bigr)^{2n} \geq 4n-1.$

Moreover, the cohomology class $\bigl(\pr_2^*([\omega]-\lambda c_1)-\pr_1^*([\omega]-\lambda c_1)\bigr)^{2n}$ is non-zero
since it contains the non-zero summand
\[
    (-1)^n\binom{2n}{n}\pr_2^*([\omega]-\lambda c_1)^n\smile\pr_1^*([\omega]-\lambda c_1)^n = (-1)^n\binom{2n}{n}([\omega]-\lambda c_1)^n\times ([\omega]-\lambda c_1)^n,
\]
by the assumption.
Threfore, we obtain $\TC(M)\geq 4n$.
By \eqref{eq:upper_bound}, we conclude that $\TC(M)$ is either $4n$ or $4n+1$.
\end{proof}

Now we are in a position to prove the main theorem (Theorem \ref{thm:main}).

\begin{proof}[Proof of Theorem \ref{thm:main}]
Let $(M,\omega)$ be a toroidally monotone or spherically monotone symplectic 4-manifold with monotonicity constant $\lambda$.
Assume $(M,\omega)$ that does not have Kodaira dimension $-\infty$.
If $(M,\omega)$ is spherically monotone, we assume that the fundamental group $\pi_1(M)$ is of type FL and every non-trivial element of the center is infinitely cyclic.
By Theorems \ref{thm:main_general_toroidally} and \ref{thm:main_general_spherically},
it is enough to show that $([\omega]-\lambda c_1)^2\in H^2(M;\RR)$ is a non-zero cohomology class,
where $c_1=c_1(M,\omega)$ is the first Chern class of $(M,\omega)$.

Let $K=-c_1\in H^2(M;\ZZ)$ denote the canonical class of $(M,\omega)$.
Since $\kappa(M,\omega)\neq -\infty$, we have $K\cdot K\geq 0$ and $K\cdot [\omega]\geq 0$.
Then,
\begin{align*}
    ([\omega]-\lambda c_1)\cdot ([\omega]-\lambda c_1) & = ([\omega]+\lambda K)\cdot ([\omega]+\lambda K) \\
    & = \int_M \omega\wedge\omega +2\lambda K\cdot [\omega] + \lambda^2 K\cdot K > 0.
\end{align*}
In particular, $([\omega]-\lambda c_1)^2$ is non-zero as a cohomology class.
This completes the proof.
\end{proof}


\subsection{Proof of Theorem \ref{thm:main_LS}}

We make use of the following estimate for category weight given by Rudyak and Oprea \cite{RO99}.

\begin{theorem}[{\cite[Theorem 4.1]{RO99}}]\label{thm:aspherical_cwgt}
    Let $X$ be a connected CW-complex and $u\in H^2(X;\RR)$ a non-zero cohomology class.
    If $u$ is aspherical, then $\cwgt(u)\geq 2$.
\end{theorem}

In view of Theorem \ref{thm:aspherical_cwgt},
we can generalize Theorem \ref{thm:aspherical_LS} as follows:

\begin{theorem}\label{thm:main_general_LS}
    Let $(M,\omega)$ be a $2n$-dimensional spherically monotone or spherically negative monotone symplectic manifold with monotonicity constant $\lambda$.
    If $([\omega]-\lambda c_1)^n$ is a non-zero cohomology class, then $\cat(M)=\dim{M}+1$.
\end{theorem}

\begin{proof}
If $(M,\omega)$ is spherically monotone or spherically negative monotone with monotonicity constant $\lambda$,
then there exists $\lambda\in\RR$ such that
\[
    [\omega]|_{h_S(\pi_2(M))} = \lambda c_1|_{h_S(\pi_2(M))}.
\]
That is,
\[
    ([\omega] - \lambda c_1)|_{h_S(\pi_2(M))} = 0.
\]
It implies that the cohomology class $[\omega] - \lambda c_1$ is aspherical.
Let $K=K(\pi_1(M),1)$ denote the Eilenberg--MacLane space of the fundamental group of $M$.
Let $f\colon M\to K$ be a continuous map which induces an isomorphism of fundamental groups.
Applying Theorem \ref{thm:aspherical_cwgt} for $[\omega] - \lambda c_1$ yields that
$\cwgt([\omega] - \lambda c_1)\geq 2$.
By Theorem \ref{thm:cateogry_weight_cup_product},
\[
    \cwgt([\omega] - \lambda c_1)^n\geq n\cdot\cwgt([\omega] - \lambda c_1)\geq 2n.
\]
Since $([\omega] - \lambda c_1)^n\neq 0$ by the assumption, we obtain $\cat(M)\geq 2n+1$.
By \eqref{eq:upper_bound_LS}, we conclude that $\cat(M) = 2n+1$.
\end{proof}

Now the proof of Theorem \ref{thm:main_LS} is almost same as in Theorem \ref{thm:main}.

\begin{proof}[Proof of Theorem \ref{thm:main_LS}]
Let $(M,\omega)$ be a spherically monotone symplectic 4-manifold with monotonicity constant $\lambda$.
If we assume that $\kappa(M,\omega)\neq -\infty$, then $([\omega]-\lambda c_1)^2$ is non-zero.
Now Theorem \ref{thm:main_general_LS} completes the proof.
\end{proof}


\section*{Acknowledgments}

The author would like to thank the anonymous referee.


\bibliographystyle{amsalpha}
\bibliography{orita_bibtex}
\end{document}